\documentclass[a4paper,10pt]{article}
\usepackage{a4}
\usepackage{amsmath}
\usepackage{amsfonts}
\usepackage{amsthm}
\usepackage{amssymb}
\usepackage{graphicx}
\usepackage{color}
 \makeatletter
\newtheorem{thm}{Theorem}[section]

\newtheorem{lemma}[thm]{Lemma}

\newtheorem{theorem}[thm]{Theorem}

\newtheorem{remark}[thm]{Remark}
\newtheorem{corollary}[thm]{Corollary}

\def \tmz {\begin{itemize}}
\def \etmz {\end{itemize}}

\def\ni{\noindent}
\numberwithin{equation}{section}
\begin{document}

\title{Universal quaternary mixed sums  involving  generalized $3$-, $4$-, $5$- and $8$-gonal  numbers via products of Ramanujan's theta functions}
\author{Nasser Abdo Saeed Bulkhali$^1$,  A. $\text{Vanitha}^{2},$ and  M. P.  $\text{Chaudhary}^{3}$ \\$^1$Department of Mathematics, Faculty of Education-Zabid,\\ Hodeidah University, Hodeidah, Yemen\\
$^{2}$  Department of  Mathematics, The National Institute of Engineering,\\ Mysuru-570008,  India\\
$^{3}$International Scientific Research and 
Welfare Organization,\\ New Delhi, India\\
\small{ \textbf{e-mails:} {nassbull@hotmail.com} ; a.vanitha4@gmail.com ; dr.m.p.chaudhary@gmail.com}}

\date{} \maketitle

\begin{abstract}
\footnotesize{ Generalized $m$-gonal numbers are those $p_m(x)= [ (m - 2)x^2 - (m - 4)x ]/2 $ where $x$ and $m$ are integers with $m \geq 3$.
If any nonnegative integer can be written in the form $ap_r(h)+bp_s(l)+cp_t(m)+dp_u(n)$, where $a,b,c,d$ are positive integers,  then we call $ap_r(h)+bp_s(l)+cp_t(m)+dp_u(n)$ a universal quaternary sum. In this paper, we determine the universality of many quaternary sums when $r,s,t,u \in \{3,4,5,8\}$, using the theory of Ramanujan's theta function identities. \\

  }
\end{abstract}
\section{Introduction}\label{Introduction}
Ramanujan's general theta function is defined by
\begin{equation}\label{2.1}f(a,b):= \sum_{n=-\infty}^{\infty}{{a}^{n(n+1)/2}}{{b}^{n(n-1)/2}},~~~|ab|<1.\end{equation}
Let $\mathbb{N}$ be denote the set of positive integers, let $\mathbb{N}_0=\mathbb{N} \cup \left\{ 0\right\} =\{0,1,2,3,\ldots\}$ and  $\mathbb{Z}$ be the set of all integers.
 \textit{A generalized $m$-gonal number} (or \textit{a generalized polygonal number of order} $m$), for $m \geq 3$, is defined by
$$ p_m(x) = \frac{(m - 2)x^2 - (m - 4)x}{2}, ~~ x\in \mathbb{Z}.$$
Observe that
$$p_3(x)=\frac{x(x+1)}{2},~~ p_4(x)=x^2, ~~ p_5(x)=\frac{x(3x-1)}{2} ~~ \text{and}~~ p_8(x)=x(3x-2), $$
are  the triangular numbers, the squares, the generalized  pentagonal numbers and the generalized octagonal  numbers, respectively. \\
For integer-valued polynomials $f_1(h), f_2(l), f_3(m), f_4(n)$, if
$$\{f_1(h) + f_2(l) +  f_3(m)+f_4(n) : h,l,m,n \in \mathbb{Z} \} = \mathbb{N}_0,$$
 then, we call $f_1(h) + f_2(l) +  f_3(m)+f_4(n)$ a universal sum over  $\mathbb{Z}$.\\
For integer-valued polynomials $f_1(x_1,\ldots, x_k )$ and $f_2(x_1,\ldots, x_k )$ with $f_i(\mathbb{Z}^k) \subseteq \mathbb{N}_0 $, if
$$\{f_1(x_1,\ldots, x_k ) : x_1,\ldots, x_k  \in \mathbb{Z}\} = \{f_2(x_1,\ldots, x_k ) : x_1,\ldots, x_k  \in \mathbb{Z} \}$$
then we say that $f_1(x_1,\ldots, x_k )$ is equivalent to $f_2(x_1,\ldots, x_k )$ and write $$f_1(x_1,\ldots, x_k ) \sim f_2(x_1,\ldots, x_k ).$$
Lagrange’s four square theorem states that every positive integer can be written as a sum of at most four integral squares.
In $1917$,  Ramanujan~\cite{Ramanujan} provided a list of all the $54$ universal sums $ap_4+bp_4+cp_4+dp_4$ with $1 \leq a \leq b \leq c \leq d$, and the list was confirmed in $1927$ by Dickson~\cite{Dickson}. In 2012  Clark,  Hicks,  Thompson and  Walters \cite{ClarkHicksThompsonWalters} redetermined the    universality of all nine diagonal forms using \textit{geometry of numbers} methods. In addition they provided a brief literature survey about the methods which was used to determine the universality of all nine diagonal forms from the set of the $54$ universal sums.  
 Sun  \cite{Sun6}, obtained a result similar to the Lagrange's theorem on sums of four squares, where he showed that every positive integer can be written as a sum of four generalized octagonal numbers one of which is odd.  Moreover, for $35$ triples $(b, c, d)$ with $1 \leq b \leq c \leq d$, he proved that any nonnegative integer can be expressed as $p_8(w) +bp_8(x) +cp_8(y) +dp_8(z)$ with $w,x,y,z \in \mathbb{Z}$.  He proved the following interesting  theorem:
\begin{theorem} \cite[Theorem 1.3]{Sun6}
$p_8+bp_8+cp_8+dp_8$ is universal over $\mathbb{Z}$ for any $(b, c, d)$ among the $33$ triples
\begin{align*}
&(1, 2, 2),~ (1, 2, 8),~ (2, 2, 4),~ (2, 4, 8),~ (2, 2, 2),~ (2, 4, 4),\\
&(1, 1, 2),~ (1, 2, 3),~ (1, 2, 5),~ (1, 2, 7),~ (1, 2, 9),~ (1, 2, 11),~ (1, 2, 13),\\
&(1, 2, 4),~(2, 3, 4),~ (2, 4, 5),~ (2, 4, 7),~ (2, 4, 9),~ (2, 4, 11),~ (2, 4, 13),\\
&(1, 1, 3),~(2, 2, 3),~ (2, 2, 6),~ (2, 3, 8),~ (1, 2, 6),~ (1, 2, 10),~ (1, 2, 12),\\
&(2, 4, 6),~(2, 4, 10),~ (2, 4, 12),~ (2, 2, 5),~ (2, 3, 5),~ (1, 3, 5).
\end{align*}
\end{theorem}
Further, he   conjectured that, the sum $p_8+bp_8+cp_8+dp_8$ is universal over $\mathbb{Z}$ if $(b, c, d)$ is among the five triples
$(1, 3, 3),~ (1, 3, 6),~ (2, 3, 6),~ (2, 3, 7),~ (2, 3, 9)$ which were confirmed by  Ju and   Oh \cite{JuOh}.\\
He  also conjectured that each $n \in \mathbb {N}_0$ can be written in the form of quaternary sum $p_5(w) + bp_5(x) + cp_5(y) + dp_5(z)$ with $w, x, y, z \in \mathbb{N}_0$, and provided a list of all the $15$ triples  $(b, c, d)$. Some of these triples were confirmed by  Meng and Sun \cite{MengSun} and   Krachun and Sun \cite{KrachunSun}.\\
In Section $4$, Ju and  Oh \cite{JuOh2} determined all quaternary universal mixed sums of generalized 4- and 8-gonal numbers including those
 94 quaternary universal sums which are due to Ramanujan \cite{Ramanujan} and Sun \cite{Sun6}. They proved that, there are exactly $547$ proper quaternary universal mixed sums of generalized 4- and 8-gonal numbers and listed them.\\
 Many other Studies have been done for such specific sum  of a generalized polygonal numbers of order $m$ with $m \geq 3$, for instance:\\
 In 2013,  Bosma and  Kane \cite{BosmaKane}  investigated the representability of positive integers as sums of triangular numbers. In particular, they  showed that $f(x_1, x_2, \ldots , x_k) = b_1p_3({x_1}) + \cdots + b_kp_3({x_k})$,  for fixed positive integers $b_1, b_2, \ldots , b_k$, represents every nonnegative integer if and only if it represents $1, 2, 4, 5$ and $8$.\\
 In $2020$,  Ju \cite{Ju} showed that the sums of the form $a_1p_5(x_1) + \cdots + a_kp_5(x_k)$  with $x_1, \ldots , x_k \in \mathbb{ Z}$ and  $a_1, \ldots, a_k \in \mathbb{ N} $, are universal if they represent  the twelve numbers $1, ~3,~ 8,~ 9,~ 11,~ 18,~ 19, ~25, ~27,~ 43, ~98, ~109.$
 Further, he proved that there are exactly $90$ quaternary proper universal sums of generalized pentagonal numbers.\\
In 2023,  Kamaraj,  Kane, and  Tomiyasu \cite{KamarajKaneTomiyasu},  considered a representation of integers as sums of the forms $a_1 p_7(x_1) +a_2 p_7(x_2) + \cdots +  a_k p_7(x_k) $, with $ a_1, a_2,\ldots, a_k \in \mathbb{Z}^+$ and $ x_1, x_2,\ldots, x_k \in \mathbb{Z}$ . In particular, they investigated the classification of such sums which are universal. They proved an explicit finite bound such that a given sum is universal if and only if it represents positive integer up to the given bound.\\
In $2024$,  Yang \cite{Yang} studied universal mixed sums of triangular numbers and squares. In particular, he proved that a sum of triangular numbers and squares is universal if and only if it represents
1, 2, 3, 4, 5, 6, 7, 8, 10, 13, 14, 15, 18, 19, 20, 23, 27, 28, 34, 41, 47, and 48.\\
 Recently, Sun \cite{Sun7} established several new results that are similar to Lagrange's four-square theorem.
For instance, Sun  showed  that any integer $n>1$ can be written as
$w(5w+1)/2+x(5x+1)/2+y(5y+1)/2+z(5z+1)/2$ with $w,x,y,z\in\mathbb Z$. In addition,
for any integer $a$ and $b$ with $a>0$, $b>-a$ and $\gcd(a,b)=1$. When $2\nmid ab$, Sun proved  that
any sufficiently large integer can be written as
$$\frac{w(aw+b)}2+\frac{x(ax+b)}2+\frac{y(ay+b)}2+\frac{z(az+b)}2$$
with $w,x,y,z$ nonnegative integers. When $2\mid a$ and $2\nmid b$,  he showed that any sufficiently large integer can be written as
$$w(aw+b)+x(ax+b)+y(ay+b)+z(az+b)$$
with $w,x,y,z$ nonnegative integers. At the end of \cite{Sun7}, Sun posed  many interesting open conjectures.

 Recently, the first author and Sun \cite{SunBulkhali} introduced a new technique to determine the universality of the sums in the form $l(al+b)/2+m(cm +d)/2+n(en +f)/2$, where $l,m,n \in \mathbb{Z}$, that are conjectured by  Sun \cite{Sun2},  using the theory of Ramanujan's theta function identities. Furthermore, Bulkhali and Sun obtained many sums of the same form that are either  universal or almost universal over $\mathbb{Z}$ (the sum is called almost universal if it represents all but finitely many positive integers). Bulkhali and Sun~\cite{SunBulkhali} derived the following novel theorem:
\begin{theorem}\label{thm1.1}
 Let $i, j, s, t, u, v \in \mathbb{Z}$ with $i + j, s + t, u + v$ positive. Suppose that
 \begin{align}\label{Athm15.11}
  f(q^i,q^j)f(q^s,q^t)f(q^u,q^v)=\sum_{r=1}^k m_{r} q^{r-1} f(q^{k a_{1 r}},q^{ka_{2r}})f(q^{k b_{1r}},q^{kb_{2r}})f(q^{k c_{1r}},q^{kc_{2r}}),
\end{align}
where  $a_{1r},a_{2r},b_{1r},b_{2r},c_{1r},c_{2r} \in \mathbb{Z}$ and $k,m_r, a_{1r}+a_{2r},b_{1r}+b_{2r},c_{1r}+c_{2r} \in \mathbb{Z}^+$.
\begin{description}
  \item[(i)] The tuple $(i + j, i - j, s + t, s - t, u + v, u - v)$ is universal if and only if
  $$(a_{1r}+a_{2r},a_{1r}-a_{2r},b_{1r}+b_{2r},b_{1r}-b_{2r},c_{1r}+c_{2r},c_{1r}-c_{2r})$$
  is universal for all $r=1,\ldots, k$.
  \item[(ii)] The tuple $(i + j, i - j, s + t, s - t, u + v, u - v)$ is almost universal if and only if the tuple
   $$(a_{1r}+a_{2r},a_{1r}-a_{2r},b_{1r}+b_{2r},b_{1r}-b_{2r},c_{1r}+c_{2r},c_{1r}-c_{2r})$$
   is almost universal for all $r=1,\ldots, k$.
\end{description}
 \end{theorem}
 Recently, Chaudhary \cite{Chaudary1} established a family of theta-function identities based upon $R_{\alpha}, R_{\beta}$ and $R_{m}$-functions related to Jacobi’s triple-product identity, see also (\cite{Chaudary2, Chaudary3, Gladun}).\\
 
 In this paper, we use a similar technique by applying  Theorem \ref{thm1.1} for the products of four theta functions and determine the universality of many quaternary mixed sums that are  involving squares, triangular numbers, generalized pentagonal numbers and generalized octagonal  numbers. The part (i) of the above theorem can be written (for the product of four theta functions) as in the following corollary:
 \begin{corollary}\label{coro1.1}
 Let $g, h, i, j, s, t, u, v \in \mathbb{Z}$ with $g+h, i + j, s + t, u + v \in \mathbb{Z}^{+}$. Suppose that
 \begin{align}\label{Q1.1}
 &f(q^g,q^h) f(q^i,q^j)f(q^s,q^t)f(q^u,q^v)\nonumber \\&=\sum_{r=1}^k m_{r} q^{r-1} f(q^{k a_{1 r}},q^{ka_{2r}})f(q^{k b_{1r}},q^{kb_{2r}})f(q^{k c_{1r}},q^{kc_{2r}})f(q^{k d_{1r}},q^{kd_{2r}}),
\end{align}
where  $a_{1r},a_{2r},b_{1r},b_{2r},c_{1r},c_{2r},d_{1r},d_{2r} \in \mathbb{Z}$ and $k,m_r, a_{1r}+a_{2r},b_{1r}+b_{2r},c_{1r}+c_{2r},d_{1r}+d_{2r} \in \mathbb{Z}^+$. The sum
\begin{equation*}
  \frac{x((g+h)x+g-h)}{2}+ \frac{y((i+j)y+i-j)}{2}+\frac{z((s+t)z+s-t)}{2}+\frac{w((u+v)w+u-v)}{2}
\end{equation*}
is universal over $\mathbb{Z}$,  if and only if the sum
\begin{align*}
  \frac{x((a_{1r}+a_{2r})x+a_{1r}-a_{2r})}{2}+ \frac{y((b_{1r}+b_{2r})y+b_{1r}-b_{2r})}{2}\\+\frac{z((c_{1r}+c_{2r})z+c_{1r}-c_{2r})}{2}+\frac{w((d_{1r}+d_{2r} )w+d_{1r}-d_{2r} )}{2}
\end{align*}
is universal over $\mathbb{Z}$.
   \end{corollary}
  In the above corollary, without loss of generality,  we assume that 
  $a_{1r}-a_{2r},b_{1r}-b_{2r},c_{1r}-c_{2r},d_{1r}-d_{2r} \in \mathbb{Z}^-\cup \{0\}$,(are not positive integers). We make use of this corollary to prove our results below. We can classify those quaternary mixed sums we determined their universality as follows:
 \begin{enumerate}
   \item The sum of the form $ap_5(h)+bp_8(l)+cp_8(m)+dp_8(n)$ is universal over $\mathbb{Z}$ if $(a,b,c,d)$ is among the quadruples
   \begin{align*}
     (4,1,1,1),~(2,1,1,2),~(4,1,1,2),~(2,1,2,2),~(1,1,1,2),~(4,1,1,3),\\
     (2,1,2,3),~(4,1,2,6),~(4,1,2,3),~(4,1,2,5),~(4,1,2,2),~(2,1,1,1).
   \end{align*}
   \item The sum of the form $ap_5(h)+bp_5(l)+cp_8(m)+dp_8(n)$ is universal over $\mathbb{Z}$ if $(a,b,c,d)$ is among the quadruples
   \begin{align*}
    (2,4,1,1),~(2,4,1,2),~(1,2,1,2),~(1,4,1,1),~(1,4,1,2),~(2,4,1,3),~(2,2,1,1).
   \end{align*}

   \item The sum of the form $ap_5(h)+bp_5(l)+cp_5(m)+dp_8(n)$ is universal over $\mathbb{Z}$ if $(a,b,c,d)$ is among the quadruples
   \begin{align*}
    (1,1,1,2),~(1,2,4,1),~(1,2,4,2),~(1,3,6,1),\\
    (1,2,3,3),~(1,1,4,2),~(1,2,2,2),~(2,2,2,1).
   \end{align*}

   \item The sum of the form $ap_3(h)+bp_5(l)+cp_8(m)+dp_8(n)$ is universal over $\mathbb{Z}$ if $(a,b,c,d)$ is among the quadruples
   \begin{align*}
    (6,2,1,1),~(6,2,1,2),~(6,2,1,3),~(6,2,1,4),~(6,2,1,5),\\
    (6,2,1,6),~(3,1,1,1),~(3,1,1,2),~(3,1,2,2),~(3,1,2,3).
   \end{align*}
    \item The sum of the form $ap_3(h)+bp_5(l)+cp_5(m)+dp_8(n)$ is universal over $\mathbb{Z}$ if $(a,b,c,d)$ is among the quadruples
    \begin{align*}
     &(3,1,4,1),~(3,1,4,2),~(3,1,4,3),~(3,1,1,2),~(3,1,2,1),~(6,1,2,1),\\
     &(6,2,2,1),~(1,1,6,1),~(1,2,6,2),~(1,4,6,1),~(2,4,12,1).
    \end{align*}
    \item The sum of the form $ap_3(h)+bp_4(l)+cp_5(m)+dp_8(n)$ is universal over $\mathbb{Z}$ if $(a,b,c,d)$ is among the quadruples
    \begin{align*}
    &(1,1,2,2),~(1,1,3,3),~(1,1,4,1),~(1,1,4,2),~(2,2,1,1),~(2,2,1,2),\\
    &(3,3,1,1),~(4,1,1,1),~(4,1,1,2),~(4,1,2,1),~(4,1,2,2),~(6,3,2,1).
     \end{align*}
    \item The sum of the form $ap_3(h)+bp_5(l)+cp_5(m)+dp_5(n)$ is universal over $\mathbb{Z}$ if $(a,b,c,d)$ is among the quadruples
    \begin{align*}
      &(1,1,1,6),~(1,2,4,6),~(2,1,4,12),~(3,1,1,4),~(1,2,6,18),\\
      &(3,1,2,2),~(4,1,2,6),~(6,1,1,2),~~(6,1,2,2).
    \end{align*}
    \item The sum of the form $ap_4(h)+bp_8(l)+cp_8(m)+dp_8(n)$ is universal over $\mathbb{Z}$ if $(a,b,c,d)$ is among the quadruples
    \begin{align*}
      &(1,1,1,2),~(3,1,1,1),~(3,1,1,2),~(3,1,1,3),~(3,1,1,4),~(3,1,1,5),\\
      &(3,1,1,6),~(3,1,2,2),~(3,1,2,3),~(6,1,1,2),~(6,1,1,3).
    \end{align*}
    \item The sum of the form $ap_3(h)+bp_3(l)+cp_4(m)+dp_5(n)$ is universal over $\mathbb{Z}$ if $(a,b,c,d)$ is among the quadruples
    \begin{align*}
      (1,1,2,3),~(1,2,2,6),~(1,2,3,1),~(1,2,3,2),~(1,3,1,2),\\
      (1,4,1,6),~(1,6,1,1),~(1,6,1,2),~(2,2,1,6),~(3,4,1,1).
    \end{align*}
    \item The sum of the form $ap_3(h)+bp_4(l)+cp_8(m)+dp_8(n)$ is universal over $\mathbb{Z}$ if $(a,b,c,d)$ is among the quadruples
    \begin{align*}
     &(1,1,1,2),~(2,2,1,1),~(2,2,1,2),~(2,2,1,3),~(4,1,1,1),~(4,1,1,2),\\
     &(4,1,1,3),~(4,1,2,2),~(4,1,2,3),~(4,1,2,5),~(4,1,2,6).
    \end{align*}
    \item The sum of the form $ap_4(h)+bp_5(l)+cp_8(m)+dp_8(n)$ is universal over $\mathbb{Z}$ if $(a,b,c,d)$ is among the quadruples
    \begin{align*}
    (3,1,1,1),~(3,2,1,1),~(3,4,1,2).
     \end{align*}
     \item The sum of the form $ap_3(h)+bp_3(l)+cp_8(m)+dp_8(n)$ is universal over $\mathbb{Z}$ if $(a,b,c,d)$ is among the quadruples
\begin{align*}
 (1,2,1,2),~(2,4,1,1),~(2,4,1,2),~(2,4,1,3).
\end{align*}
    \item The sum of the form $ap_3(h)+bp_3(l)+cp_5(m)+dp_5(n)$ is universal over $\mathbb{Z}$ if $(a,b,c,d)$ is among the quadruples
\begin{align*}
  (1,2,2,4),~(1,3,2,6),~(1,4,1,2),~(1,6,2,6),\\
  (2,4,1,2),~(2,4,1,4),~(3,3,1,1),~(3,6,1,2).
\end{align*}
 \item The sum of the form $p_3(h)+bp_3(l)+cp_3(m)+dp_5(n)$ is universal over $\mathbb{Z}$ if $(b,c,d)$ is among the triples
\begin{align*}
(1,2,6),~(1,4,3),~(2,3,2),~(2,3,4),~(2,6,2).
\end{align*}
 \item The sum of the form $ap_3(h)+bp_3(l)+cp_5(m)+dp_8(n)$ is universal over $\mathbb{Z}$ if $(a,b,c,d)$ is among the quadruples
\begin{align*}
(1,2,1,1),~(1,2,2,2),~(1,2,4,1),~(1,2,4,2),~(2,4,1,1),~(2,4,4,1).
\end{align*}
 \item The sum of the form $ap_4(h)+bp_5(l)+cp_5(m)+dp_8(n)$ is universal over $\mathbb{Z}$ if $(a,b,c,d)$ is among the quadruples
\begin{align*}
 (1,2,3,2),~(2,4,6,1),~(3,1,1,1),~(3,1,2,1).
\end{align*}
  \item The sum of the form $ap_3(h)+bp_4(l)+cp_5(m)+dp_5(n)$ is universal over $\mathbb{Z}$ if $(a,b,c,d)$ is among the quadruples
\begin{align*}
 (1,3,2,6),~(1,1,1,2),~(3,3,1,1),~(4,1,1,2),~(6,1,1,3),~(6,3,1,1).
\end{align*}
  \item The sum of the form $p_8(h)+bp_8(l)+cp_8(m)+dp_8(n)$ is universal over $\mathbb{Z}$ if $(b,c,d)$ is among the triples
\begin{align*}
(1,1,1),~ (1,1,2),~(1,2,2),~(1,2,3),~(2,2,2),~(2,2,3),~(2,2,5),~(2,2,6).
\end{align*}
\begin{remark}
	The universality of these sums have been determined by Sun \cite{Sun6} using the theory  of quaternary quadratic forms. Here we redetermine using  the theory of Ramanujan's theta function identities.
\end{remark}
 \item The sum of the form $ap_4(h)+bp_4(l)+cp_8(m)+dp_8(n)$ is universal over $\mathbb{Z}$ if $(a,b,c,d)$ is among the quadruples
\begin{align*}
(1,3,1,2),~(1,6,1,1),~(3,3,1,1).
\end{align*}
\begin{remark}
The universality of the above three sums have been determined by Ju and Oh \cite{JuOh2}. Here we redetermine using our approaches.
\end{remark}
 \item The sum of the form $p_3(h)+bp_3(l)+cp_4(m)+p_8(n)$ is universal over $\mathbb{Z}$ if $(b,c)$ is among the ordered pairs
\begin{align*}
(2,3),~(3,1),~(6,1).
\end{align*}
 \item The sum of the form $p_5(h)+p_5(l)+cp_5(m)+dp_8(n)$ is universal over $\mathbb{Z}$ if $(c,d)$ is among the ordered pairs
\begin{align*}
(1,4),~(3,6).
\end{align*}
\begin{remark}
{Ju} \cite{Ju} showed that the above two sums are universal.
\end{remark}
\item The sum of the form $ap_3(h)+bp_3(l)+cp_3(m)+dp_8(n)$ is universal over $\mathbb{Z}$ if $(a,b,c,d)$ is among the quadruples
\begin{align*}
(1,2,3,1),~(1,2,3,2).
\end{align*}
\item The sum of the form $ap_3(h)+bp_4(l)+cp_4(m)+dp_8(n)$ is universal over $\mathbb{Z}$ if $(a,b,c,d)$ is among the quadruples
\begin{align*}
(1,1,3,1),~(4,1,3,2).
\end{align*}
\item In addition,  the following sums are universal
$$p_4+2p_5+3p_5+4p_5, ~~ p_3+2p_5+6p_5+18p_5~~ \text{and}~~ 4p_3+p_4+2p_4+6p_5.$$
 \end{enumerate}

The paper is organised as follows: In Section \ref{Preliminary}, we collect the necessary definitions and lammas and we derive some results needed in the rest of the paper. In Section \ref{MainResults}, we present and prove many beautiful theorems.
\section{Preliminary results}\label{Preliminary}
The function $f(a,b)$ satisfies the following
basic properties \cite{Adiga3}:
\begin{align}\label{2.3}f(a,b)& = f(b,a)\\
\label{2.4}f(1,a)& = 2f(a,a^{3}).
\end{align}
Ramanujan has defined the following  two  special cases of
(\ref{2.1}) \cite[Entry 22]{Adiga3} for $|q|<1$:
\begin{align}\label{2.7}\varphi(q)&:= f(q,q) =  \sum_{n=-\infty}^{\infty} q^{n^2},\\
\label{2.8}\psi(q)&:= f(q,q^3) =  \sum_{n=-\infty}^{\infty} {q^{n(2n-1)}}=  \sum_{n=0}^{\infty} {q^{n(n+1)/2}}.
\end{align}
Also, following Ramanujan,  we define
\begin{align}\label{X(q)}
 X(q):&= f(q,q^2) =  \sum_{x=-\infty}^{\infty} {q^{x(3x-1)/2}},\\
Y(q):&= f(q,q^5) =  \sum_{x=-\infty}^{\infty} {q^{x(3x-2)}}. \label{Y(q)}
\end{align}
The above defined four functions can be considered as the generating functions of the squares $p_4(x)$, the triangular numbers $p_3(x)$, the generalized  pentagonal numbers $p_5(x)$ and the generalized octagonal  numbers $p_8(x)$, respectively. In general, the  theta function $f(q, q^{m-3})$ is the generating function of the generalized $m$-gonal number  $p_m(x)$ with $m \geq 3$. \\
Since theta function is symmetric (by \eqref{2.3}),  it follows that $f(q, q^{m-3})=f(q^{m-3},q)$ and so
\begin{equation}\label{sym}
  \{p_m(x):m,x \in \mathbb{Z}, m \geq 3 \}=  \{p_m(-x):m, x \in \mathbb{Z}, m \geq 3 \},
\end{equation}
which means that $p_m(x) \sim p_m(-x)$.\\
Using \eqref{2.8} and \eqref{2.4}, we observe that
\begin{equation}\label{2.8.1}
  \sum_{n=-\infty}^{\infty} {q^{n(n+1)/2}}=2\, \sum_{n=0}^{\infty} {q^{n(n+1)/2}} =2\, \sum_{n=-\infty}^{\infty} {q^{n(2n-1)}}.
\end{equation}
So that, using  \eqref{2.8.1}, we deduce
\begin{equation}\label{2.8.2}
  \{ p_3(x):x\in \mathbb{Z} \} =  \{x(2x-1):x\in \mathbb{Z}   \} =\{ p_6(x):x\in \mathbb{Z} \}.
\end{equation}
which means that $ p_3(x) \sim  p_6(x)$.\\
The following formula is due to Ramanujan~\cite[Entry 31]{Adiga3}:\\
Let $U_n=a^{n(n+1)/2}b^{n(n-1)/2}$ and $V_n=a^{n(n-1)/2}b^{n(n+1)/2}$ for each integer $n$. Then
\begin{equation}\label{RamIden}
  f(U_1,V_1)=\sum_{r=0}^{n-1}U_r f \left( \frac{U_{n+r}}{U_r},\frac{V_{n-r}}{U_r}  \right).
\end{equation}
The following identity is an easy consequence of \eqref{RamIden} when $n=2$:
\begin{equation}\label{P2.10}
    f(a,b)=f\left(a^3b,ab^3\right) + a f\left(b/a,a^5b^3\right).
\end{equation}
Using  \eqref{P2.10}, it is easy to verify the following lemma:
\begin{lemma} We have
\begin{align}\label{varphi}
  \varphi(q)=&\varphi\left(q^4\right)+2q\psi\left(q^8\right),\\
  \label{AAthm16.3}  \psi(q)=&f\left(q^6,q^{10}\right)+qf\left(q^2,q^{14}\right),\\
  \varphi(q)=&\varphi \left(q^9 \right)+2\,q\, Y \left(q^3\right), \label{t-s1varphi}\\
\psi(q)=&X \left(q^3 \right)+ q \, \psi \left(q^9 \right), \label{t-s2psi}\\
  \label{Y}  Y(q)=&X\left(q^8\right)+qY\left(q^4\right).
\end{align}
\end{lemma}
\noindent Using  \cite[Entry 29)]{Adiga3}, we find that
\begin{lemma} If $ab=cd$, then
\begin{equation}\label{RamIden2}
  f(a,b)f(c,d)=f\left(ac,bd\right)f\left(ad,bc\right)+a\,f\left(\frac{b}{c},ac^2d\right)f\left(\frac{b}{d},acd^2\right).
\end{equation}
\end{lemma}
\noindent From the identity \eqref{RamIden2}, we obtain
\begin{align}\label{varphi2}
\varphi^2(q)=&\varphi^2(q^2)+4q\psi^2(q^4),\\
\psi^2(q)=&\varphi\left(q^4\right)\psi(q^2)+2q\psi(q^2)\psi\left(q^8\right),\label{psi2}\\
Y^2(q)=&\varphi\left(q^6\right)Y(q^2)+2q\psi\left(q^{12}\right) X(q^4),\label{Y2}\\
\label{varphi3Y1}  \varphi(q^3)Y(q)=&X^2(q^4)+qY^2(q^2).
\end{align}
 From \cite[Corollary 5]{Adiga2}, we have
\begin{equation}\label{psi1psi3}
  \psi(q)\psi(q^3)= \varphi(q^6)\psi(q^4)+q\varphi(q^2)\psi(q^{12}).
\end{equation}
In \cite[Eq. (14)]{Adiga2}, setting $k=3$,  $l=2$, $g=10$, $h=2$, $u=5$, $v=1$ and $\varepsilon_1=\varepsilon_2=1$,  we obtain
\begin{equation}\label{Y1Y2}
  Y(q)Y(q^2)=\varphi(q^{18})Y(q^3)+q\varphi(q^9)Y(q^6)+q^2 Y(q^3)Y(q^6).
\end{equation}
In \cite[Eq. (50)]{Adiga2}, using the substitution $k=3$, $l=2$ and $\varepsilon_1=\varepsilon_2=1$, we find that
\begin{equation}\label{X1X2}
  X(q)X(q^2)=\varphi(q^9)X(q^3)+qX(q^3)Y(q^3)+2q^2\psi(q^9)Y(q^3).
\end{equation}
 In \cite[Eq. (15)]{Adiga2}, putting $k=3$,  $l=2$, $g=5$, $h=1$, $u=2$, $v=1$ and $\varepsilon_1=\varepsilon_2=1$,  we deduce
\begin{equation}\label{X1Y1}
  X(q)Y(q)=X(q^3)X(q^6)+2q\psi(q^9)X(q^6)+2q^2\psi(q^{18})X(q^3).
\end{equation}
The following helpful lemma is due to Sun \cite{Sun3}, here we state as:
\begin{lemma}\cite[Theorem 1.9]{Sun3} \label{SunLemma}  For $h,l \in \mathbb{Z}$, we have
\begin{description}
  \item[(i)] For any $a \in \mathbb{Z}^+$ and $b \in \mathbb {N}$ with $b \leq a/2$, we have
  \begin{equation}\label{2.10P1ab}
  h(ah + b) + l(al + a - b) \sim ap_3(h)+\frac{l(al+a-2b)}{2}.
  \end{equation}
  \item[(ii)] We have
\begin{align}
p_4(h)+p_3(l) & \sim p_5(h) + 2p_5(l), \label{SunLemma1}\\
p_3(h)+2p_3(l) &\sim p_5(h)+p_8(l),\label{SunLemma2}\\
p_4(h)+4p_3(l) &\sim 4p_5(h)+p_8(l).\label{SunLemma3}
\end{align}
\end{description}
\end{lemma}
Using \eqref{2.10P1ab}, it is easy to verify the following:
\begin{align}
  p_3(h)+p_3(l) & \sim p_4(h)+2p_3(l), \label{eqv1}\\
  2p_5(h)+p_8(l) & \sim 3p_3(h)+p_5(l).\label{eqv2}
\end{align}
In \cite[Eq. (1)]{WuSun}  Wu and  Sun proved that
\begin{equation}\label{WuSunequiv}
  p_3(h)+p_5(l)\sim p_5(h)+3p_5(l)
\end{equation}
\section{Main results} \label{MainResults}
In this section, we present many beautiful theorems and prove them.
For convenience, in the rest of the paper,  we use the notation $ap_{r}+bp_{s}+cp_{t}+dp_{u}$ instead of
$ap_{r}(h)+bp_{s}(l)+cp_{t}(m)+dp_{u}(n)$.
\begin{theorem}\label{thm1} The following sums:
\begin{center}
\begin{tabular}{ccc}
$2p_5+4p_5+p_8+p_8,$&$ 4p_5+p_8+p_8+p_8,$& $2p_5+p_8+p_8+2p_8,$\\
$p_8+p_8+p_8+2p_8,$&$4p_5+p_8+2p_8+2p_8,$&$p_8+2p_8+2p_8+2p_8,$ \\
$3p_4+p_5+2p_5+p_8,$&$6p_3+p_5+2p_5+2p_5,$&$3p_4+p_5+p_8+p_8,$\\
$6p_3+p_5+2p_5+p_8,$&$ 3p_4+2p_5+p_8+p_8,$&$3p_4+p_8+p_8+p_8,$\\
$6p_3+2p_5+2p_5+p_8,$&$ 6p_3+2p_5+p_8+p_8,$&$3p_3+p_5+p_5+4p_5,$\\
$3p_3+p_5+p_5+2p_8,$&$ 3p_3+3p_4+p_5+p_8,$&$3p_3+6p_3+p_5+2p_5$,\\
$3p_4+p_8+p_8+2p_8,$&$ 6p_3+2p_5+p_8+2p_8,$&$3p_3+p_5+4p_5+p_8,$\\
$ 3p_3+p_5+p_8+2p_8,$&$ p_3+p_4+3p_4+p_8, $&$ p_3+6p_3+p_4+2p_5,$\\
$p_3+2p_3+3p_4+p_8,$&$  p_3+6p_3+p_4+p_8,$&$ 2p_5+4p_5+p_8+2p_8,$\\
$ 4p_5+p_8+p_8+2p_8,$&$2p_5+p_8+2p_8+2p_8,$&$  p_8+p_8+2p_8+2p_8,$\\
$ p_5+2p_5+4p_5+p_8,$&$ p_5+2p_5+p_8+2p_8,$&$ p_5+4p_5+p_8+p_8,$\\
$ p_5+p_8+p_8+2p_8,$&$ p_3+p_4+4p_5+p_8,$&$p_3+p_4+p_8+2p_8,$\\
$p_3+2p_3+2p_5+4p_5, $&$ p_3+2p_3+2p_5+2p_8,$&$2p_3+2p_4+p_5+p_8,$\\
$ 2p_3+4p_3+p_5+p_8,$&$ p_3+2p_5+4p_5+6p_5,$&$ p_3+2p_5+6p_5+2p_8,$\\
$ 2p_5+4p_5+p_8+3p_8, $&$  4p_5+p_8+p_8+3p_8,$&$ 2p_5+p_8+2p_8+3p_8,$\\
$ p_8+p_8+2p_8+3p_8,$&$ 4p_5+p_8+2p_8+6p_8,$&$ p_8+2p_8+2p_8+6p_8,$\\
$ 3p_4+p_8+p_8+3p_8, $&$ 6p_3+2p_5+p_8+3p_8,$&$ 3p_4+p_8+p_8+4p_8,$\\
$6 p_3+2p_5+p_8+4p_8, $&$ 3p_3+p_5+4p_5+2p_8,$&$3 p_3+p_5+2p_8+2p_8,$\\
$ 3p_4+p_8+p_8+5p_8,$&$6p_3+2p_5+p_8+5p_8,$&$ 3p_4+p_8+p_8+6p_8,$\\
$ 6p_3+2p_5+p_8+6p_8,$&$ 3p_3+p_5+4p_5+3p_8,$&$ 3p_3+p_5+2p_8+3p_8,$\\
$ 4p_5+p_8+2p_8+3p_8,$&$ p_8+2p_8+2p_8+3p_8,$&$ 4p_5+p_8+2p_8+5p_8,$\\
$ p_8+2p_8+2p_8+5p_8,$&$2p_3+2p_4+p_8+3p_8,$ &$ 2p_3+4p_3+p_8+3p_8,$\\
\end{tabular}
\end{center}
are universal over $\mathbb{Z}$.
\end{theorem}
\begin{proof} Using \eqref{Y}, we find that
\begin{align}\label{Q1}
 Y(q)Y(q^2)Y^2(q^4)=&X(q^8)X(q^{16})Y^2(q^4)+qX(q^{16})Y^3(q^4)+q^2X(q^8)Y^2(q^4)Y(q^8)\nonumber \\&+q^3Y^3(q^4)Y(q^8).
\end{align}
Sun \cite{Sun6} proved that the sum $p_8+2p_8+4p_8+4p_8$ is universal. Thus the identity
  \eqref{Q1} with the  help of Corollary \ref{coro1.1}, yields the universality of the
  sums $2p_5+4p_5+p_8+p_8,$  $4p_5+p_8+p_8+p_8,$ $2p_5+p_8+p_8+2p_8$ and
$p_8+p_8+p_8+2p_8$.\\
Again,  using \eqref{Y}, we have
\begin{align}\label{Q1a}
 Y(q)Y(q^2)Y^2(q^4)=X(q^8)Y(q^{2})Y^2(q^4)+qY(q^2)Y^3(q^4).
\end{align}
Thus the universality of the sum $p_8+2p_8+4p_8+4p_8$ and the identity
  \eqref{Q1a} with the  help of Corollary \ref{coro1.1}, imply  the universality of the
  sums $4p_5+p_8+2p_8+2p_8$ and $p_8+2p_8+2p_8+2p_8$.
Using \eqref{Y2}, we obtain
\begin{equation}\label{Q2}
  X(q^2)X(q^{4})Y^2(q)=\varphi\left(q^6\right) X(q^2)X(q^{4}) Y(q^2)+2q\psi\left(q^{12}\right)X(q^2) X^2(q^4).
\end{equation}
Since $2p_5+4p_5+p_8+p_8$  is universal. Thus the identity
  \eqref{Q2} with the  help of Corollary \ref{coro1.1}, implies the universality of the sums
$3p_4+p_5+2p_5+p_8$ and $6p_3+p_5+2p_5+2p_5$.\\
From \eqref{Y2}, we have
\begin{equation}\label{Q3}
  X(q^2)Y^2(q)Y(q^2)=\varphi\left(q^6\right) X(q^2) Y^2(q^2)+2q\psi\left(q^{12}\right)X(q^2) X(q^4)Y(q^2).
\end{equation}
Since $2p_5+p_8+p_8+2p_8$  is universal.  Thus the identity
  \eqref{Q3} with the  help of Corollary \ref{coro1.1}, implies that the sums
$3p_4+p_5+p_8+p_8$ and  $6p_3+p_5+2p_5+p_8$ are universal over $\mathbb{Z}$.\\
Using \eqref{Y} and \eqref{Y2}, we find that
\begin{align}\label{Q4}
  Y(q)Y^2(q^2)Y(q^4)=& \varphi\left(q^{12}\right) X(q^8) Y^2(q^4)+q \varphi\left(q^{12}\right) Y^3(q^4)\nonumber\\
  &+2q^2\psi\left(q^{24}\right)X^2(q^8) Y(q^4)+2q^3\psi\left(q^{24}\right)X(q^8) Y^2(q^4).
\end{align}
Sun \cite{Sun6} showed that the sum $p_8+2p_8+2p_8+4p_8$ is universal.
Thus the identity \eqref{Q4} with the  help of Corollary \ref{coro1.1}, implies that the sums
$3p_4+2p_5+p_8+p_8,$ $3p_4+p_8+p_8+p_8,$  $6p_3+2p_5+2p_5+p_8$ and  $6p_3+2p_5+p_8+p_8$ are universal.\\
Using \eqref{Y}, we get
\begin{equation}\label{Q5}
  \psi\left(q^{6}\right)X^2(q^2) Y(q)= \psi\left(q^{6}\right)X^2(q^2)X\left(q^8\right)+q \psi\left(q^{6}\right)X^2(q^2)Y\left(q^4\right).
\end{equation}
Since $6p_3+2p_5+2p_5+p_8$  is universal.  Thus the identity
  \eqref{Q5} with the  help of Corollary \ref{coro1.1}, implies
universality of the sums $3p_3+p_5+p_5+4p_5$  and $3p_3+p_5+p_5+2p_8$.\\
From \eqref{Y2}, we have
\begin{equation}\label{Q6}
\psi\left(q^{6}\right)X(q^2)Y^2(q)=\varphi\left(q^6\right)\psi\left(q^{6}\right)X(q^2)Y(q^2)+2q\psi\left(q^{6}\right)\psi\left(q^{12}\right)X(q^2) X(q^4).
\end{equation}
 Since $6p_3+2p_5+p_8+p_8$ is universal. Thus the identity
  \eqref{Q6} implies that the sums $3p_3+3p_4+p_5+p_8$ and $3p_3+6p_3+p_5+2p_5$ are universal.\\
Using  \eqref{Y2}, we have
\begin{equation}\label{Q7}
Y^2(q)Y(q^2)Y(q^4)=\varphi\left(q^6\right)Y^2(q^2)Y(q^4)+2q\psi\left(q^{12}\right) X(q^4)Y(q^2)Y(q^4).
\end{equation}
Sun \cite{Sun6} proved that the sum $p_8+p_8+2p_8+4p_8$ is universal.
Thus  \eqref{Q7} yields the universality of the sums  $3p_4+p_8+p_8+2p_8$ and $6p_3+2p_5+p_8+2p_8$.\\
Using \eqref{Y}, we find that
\begin{equation}\label{Q8}
\psi\left(q^{6}\right)X(q^2) Y(q) Y(q^2) = \psi\left(q^{6}\right)X(q^2)X\left(q^8\right)Y(q^2)+q \psi\left(q^{6}\right)X(q^2)Y(q^2)Y\left(q^4\right).
\end{equation}
So that \eqref{Q8} and  the universality of the sum $6p_3+2p_5+p_8+2p_8$ yield the universality of the sums
$3p_3+p_5+4p_5+p_8$ and  $3p_3+p_5+p_8+2p_8$.\\
From  \eqref{Y2}, we obtain
\begin{equation}\label{Q9}
\varphi\left(q^2\right) \psi\left(q^{2}\right) Y^2(q)= \varphi\left(q^2\right)\varphi\left(q^6\right)\psi\left(q^{2}\right)Y(q^2)+2q\varphi\left(q^2\right) \psi\left(q^{2}\right) \psi\left(q^{12}\right) X(q^4).
\end{equation}
Since  $3p_3+p_5+4p_5+p_8$  is universal and in view of \eqref{eqv2} and \eqref{SunLemma1} we have $3p_3+p_5+4p_5+p_8 \sim  2p_5+4p_5+p_8+p_8 \sim 2p_3+2p_4+p_8+p_8$ which means that the sum  $2p_3+2p_4+p_8+p_8$  is universal. Thus \eqref{Q9} yields the universality of the sums
$ p_3+p_4+3p_4+p_8$ and  $p_3+6p_3+p_4+2p_5$.\\
Using \eqref{psi1psi3}, we deduce
\begin{equation}\label{Q10}
  \psi(q) \psi(q^2) \psi(q^3) Y(q^2)= \varphi(q^6)\psi(q^2) \psi(q^4)Y(q^2)+q\varphi(q^2)\psi(q^2)\psi(q^{12})Y(q^2).
\end{equation}
Since $3p_3+p_5+p_8+2p_8$ is universal and in view of \eqref{SunLemma2} we have $3p_3+p_5+p_8+2p_8 \sim p_3+2p_3+ 3p_3+2p_8$
and hence $p_3+2p_3+ 3p_3+2p_8$ is universal. Thus \eqref{Q10} yields the universality of the sums
$p_3+2p_3+3p_4+p_8$ and  $ p_3+6p_3+p_4+p_8$.\\
Using \eqref{Y}, we find that
\begin{align}\label{Q11}
Y(q)Y(q^2)Y(q^4)Y(q^8)=& X\left(q^{8}\right)X(q^{16}) Y(q^4) Y(q^8)+qX(q^{16}) Y^2(q^4) Y(q^8)\nonumber\\
&+ q^2 X(q^{8}) Y(q^4) Y^2(q^8)+ q^3  Y^2(q^4) Y^2(q^8).
\end{align}
Sun \cite{Sun6} proved that the sum $p_8+2p_8+4p_8+8p_8$ is universal.
Thus the identity  \eqref{Q11} yields the universality of the sums
$2p_5+4p_5+p_8+2p_8,$  $4p_5+p_8+p_8+2p_8,$  $2p_5+p_8+2p_8+2p_8$  and  $p_8+p_8+2p_8+2p_8$.\\
From \eqref{Y}, we obtain
\begin{align}\label{Q12}
X\left(q^{2}\right)X(q^{4}) Y(q) Y(q^2)=X\left(q^{2}\right)X(q^{4}) X(q^8) Y(q^2)+q X\left(q^{2}\right)X(q^{4}) Y(q^2) Y(q^4).
\end{align}
Since $2p_5+4p_5+p_8+2p_8$  is universal over $\mathbb{Z}$. So that \eqref{Q12} yields the universality of the sums
$p_5+2p_5+4p_5+p_8$ and $p_5+2p_5+p_8+2p_8.$\\
From \eqref{Y}, we obtain
\begin{align}\label{Q13}
X(q^{2}) Y(q) Y^2(q^2)=X(q^{2}) X(q^8) Y^2(q^2)+qX(q^{2})Y^2(q^2) Y(q^4).
\end{align}
Since $2p_5+p_8+2p_8+2p_8$ is universal. Thus \eqref{Q12} with the  help of Corollary \ref{coro1.1}, implies that the sums
$p_5+4p_5+p_8+p_8$ and  $p_5+p_8+p_8+2p_8$ are universal. \\
Using \eqref{Y}, we find that
\begin{equation}\label{Q15}
  \varphi\left(q^2\right) \psi\left(q^{2}\right) Y(q) Y(q^2)=\varphi\left(q^2\right) \psi\left(q^{2}\right)X\left(q^8\right)Y(q^2)+q\varphi\left(q^2\right) \psi\left(q^{2}\right)Y(q^2)Y\left(q^4\right).
\end{equation}
As $p_5+4p_5+p_8+p_8$ is universal, in  view of   \eqref{SunLemma1}, we have  $p_5+4p_5+p_8+p_8 \sim 2p_3+2p_4+p_8+2p_8, $
this means that the sum $2p_3+2p_4+p_8+2p_8 $ is universal. Thus \eqref{Q15} yields the universality of the sums
$p_3+p_4+4p_5+p_8$ and $p_3+p_4+p_8+2p_8.$\\
Using \eqref{Y}, we find that
\begin{equation}\label{Q16}
\psi\left(q^{2}\right)\psi\left(q^{4}\right)X(q^4)Y(q)=\psi\left(q^{2}\right)\psi\left(q^{4}\right)X(q^4)X\left(q^8\right)+q\psi\left(q^{2}\right)\psi\left(q^{4}\right)X(q^4)Y\left(q^4\right).
\end{equation}
As $p_5+4p_5+p_8+p_8$ is universal, in  view of   \eqref{SunLemma2}, we have  $p_5+4p_5+p_8+p_8 \sim 2p_3+4p_4+4p_5+2p_8, $
this means that the sum $2p_3+4p_4+4p_5+2p_8 $ is universal. Thus \eqref{Q16} implies the universality of the sums
$p_3+2p_3+2p_5+4p_5$ and $p_3+2p_3+2p_5+2p_8.$\\
In a similar way, in view of \eqref{SunLemma3} and \eqref{WuSunequiv}, we have $2p_5+4p_5+p_8+2p_8 \sim 4p_3+p_4+2p_5+2p_8$  and   $2p_3+4p_3+4p_5+p_8 \sim 2p_3+4p_5+12p_5+p_8$ and so that $4p_3+p_4+2p_5+2p_8$  and $2p_3+4p_5+12p_5+p_8$  are universal. Using \eqref{Y}, twice we deduce the universality of the sums
$2p_3+2p_4+p_5+p_8,$ $2p_3+4p_3+p_5+p_8,$  $p_3+2p_5+4p_5+6p_5$ and $p_3+2p_5+6p_5+2p_8.$\\
Using \eqref{Y}, we obtain
\begin{align}\label{Q17}
 Y(q)Y(q^2)Y(q^4)Y(q^{12})=& X\left(q^{8}\right)X(q^{16}) Y(q^4) Y(q^{12})+q X(q^{16}) Y^2(q^4) Y(q^{12})\nonumber\\
&+ q^2 X(q^{8}) Y(q^4) Y(q^8)Y(q^{12})+ q^3  Y^2(q^4) Y(q^8)Y(q^{12}).
\end{align}
Since $p_8+2p_8+4p_8+12p_8$ is universal (see \cite{Sun6}). Thus the identity
  \eqref{Q17} implies that the sums $2p_5+4p_5+p_8+3p_8,$  $4p_5+p_8+p_8+3p_8,$
 $2p_5+p_8+2p_8+3p_8$ and  $p_8+p_8+2p_8+3p_8$ are universal. Again, using \eqref{Y}, the universality of $p_8+2p_8+4p_8+12p_8$
yields the universality of the sums $4p_5+p_8+2p_8+6p_8$ and $p_8+2p_8+2p_8+6p_8$.\\
Sun \cite{Sun6} proved that the sums $p_8+p_8+2p_8+6p_8$  and $p_8+p_8+2p_8+p_8$ are universal.
 Using \eqref{Y2}, twice we deduce the universality of the sums $3p_4+p_8+p_8+3p_8,$ $6p_3+2p_5+p_8+3p_8,$
$3p_4+p_8+p_8+4p_8$  and $6 p_3+2p_5+p_8+4p_8$. Similarly, using \eqref{Y}, the universality of $6 p_3+2p_5+p_8+4p_8$
implies the universality of the sums $3p_3+p_5+4p_5+2p_8$ and $3 p_3+p_5+2p_8+2p_8.$\\
Sun \cite{Sun6} determined the universality of   the sums $p_8+p_8+2p_8+10p_8$  and $p_8+p_8+2p_8+12p_8$. Thus the identity
  \eqref{Y2} implies the universality of the sums $3p_4+p_8+p_8+5p_8,$
$6p_3+2p_5+p_8+5p_8,$ $3p_4+p_8+p_8+6p_8$ and  $6p_3+2p_5+p_8+6p_8.$ The universality of $6p_3+2p_5+p_8+6p_8$ with the help of
\eqref{Y}, yields the universality of the sums  $3p_3+p_5+4p_5+3p_8$ and  $3p_3+p_5+2p_8+3p_8$.\\
Sun \cite{Sun6} determined the universality of   the sums $p_8+2p_8+4p_8+6p_8$  and $p_8+2p_8+4p_8+10p_8$.
 Thus the identity  \eqref{Y} implies the universality of the sums $4p_5+p_8+2p_8+3p_8,$
$ p_8+2p_8+2p_8+3p_8,$ $4p_5+p_8+2p_8+5p_8$ and $ p_8+2p_8+2p_8+5p_8.$\\
As $4p_5+p_8+2p_8+6p_8$ is universal and  $4p_5+p_8+2p_8+6p_8 \sim p_4+4p_3+2p_8+6p_8$ (by \eqref{SunLemma3}).
Thus the universality of the sum $p_4+4p_3+2p_8+6p_8$ with make use of \eqref{varphi}, we deduce that the sums $2p_3+2p_4+p_8+3p_8$ and  $2p_3+4p_3+p_8+3p_8$ are universal over $\mathbb{Z}$. This completes the proof of the theorem.
\end{proof}
\begin{theorem}\label{thm2} The following sums
 \begin{center}
\begin{tabular}{ccc}
$6p_4+p_8+p_8+2p_8,$& $3p_4+p_8+2p_8+2p_8,$&$6p_4+p_8+p_8+3p_8,$\\
$3p_4+p_8+2p_8+3p_8,$&$p_4+6p_4+p_8+p_8,$& $p_4+3p_4+p_8+2p_8,$\\
 $p_4+p_8+p_8+2p_8,$&$p_3+p_4+p_5+2p_5,$ & $p_3+3p_3+p_4+2p_5,$\\
  $p_3+6p_3+p_4+p_5,$&$p_3+2p_3+3p_4+p_5,$& $p_3+2p_3+p_5+p_8,$\\
  $p_3+2p_3+3p_3+p_8,$&&\\
\end{tabular}
\end{center}
are universal over $\mathbb{Z}$.
\end{theorem}
\begin{proof}Using \eqref{Y1Y2}, we obtain
\begin{align}\label{Q18}
 Y(q)Y(q^2)Y(q^3)Y(q^6)=&\varphi(q^{18})Y^2(q^3)Y(q^6)+q\varphi(q^9)Y(q^3)Y^2(q^6)\nonumber\\
&+q^2 Y^2(q^3)Y^2(q^6).
\end{align}
Ju and  Oh \cite{JuOh} proved that the sum $p_8+2p_8+3p_8+6p_8$ is universal.
Thus \eqref{Q18} implies the universality of the sums
$6p_4+p_8+p_8+2p_8,$ $3p_4+p_8+2p_8+2p_8$ and  $p_8+p_8+2p_8+2p_8.$ In a similar way, the
universality of the sum $p_8+2p_8+3p_8+9p_8$ (by Ju and  Oh \cite{JuOh}) with the help of
\eqref{Y1Y2}, implies the universality of the sums $6p_4+p_8+p_8+3p_8,$ $3p_4+p_8+2p_8+3p_8$ and  $p_8+p_8+2p_8+3p_8.$ Similarly,
since $3p_4+p_8+2p_8+3p_8$ is universal then the sums $p_4+6p_4+p_8+p_8,$ $p_4+3p_4+p_8+2p_8$ and  $p_4+p_8+p_8+2p_8$ are universal over $\mathbb{Z}$.\\
Using \eqref{X1Y1}, we have
\begin{align}\label{Q19}
\varphi(q^{3}) \psi(q^{3}) X(q)Y(q)= &\varphi(q^{3}) \psi(q^{3}) X(q^3)X(q^6)+2q \varphi(q^{3}) \psi(q^{3})\psi(q^9)X(q^6)\nonumber\\
&+2q^2\varphi(q^{3}) \psi(q^{3})\psi(q^{18})X(q^3).
\end{align}
From Theorem \ref{thm1}, the sum  $3p_3+3p_4+p_5+p_8$  is universal.  Thus \eqref{Q19}
yields the universality of the sums $p_3+p_4+p_5+2p_5,$ $p_3+3p_3+p_4+2p_5$ and $p_3+6p_3+p_4+p_5.$
In a similar way, the universality of the sum  $3p_3+6p_3+p_5+2p_5$ (by Theorem \ref{thm1}), with the help of
\eqref{X1X2}, implies the universality of the sums  $p_3+2p_3+3p_4+p_5,$  $p_3+2p_3+p_5+p_8$  and $ p_3+2p_3+3p_3+p_8.$ This completes the proof of the theorem.
\end{proof}
\begin{theorem}\label{thm3} The following sums
\begin{center}
\begin{tabular}{ccc}
$2p_5+2p_5+p_8+p_8,$&  $p_8+p_8+p_8+p_8,$&$ 3p_4+3p_4+p_8+p_8,$\\
$6p_3+3p_4+2p_5+p_8,$& $3p_3+p_5+2p_5+2p_5,$&$3p_3+p_5+p_8+p_8,$\\
$3p_4+p_5+p_5+p_8,$ & $6p_3+p_5+p_5+2p_5,$& $p_5+p_5+p_5+4p_5,$\\
$p_5+p_5+p_5+2p_8,$& $p_3+p_3+2p_4+3p_5,$ & $p_3+p_3+4p_3+3p_5,$\\
 $p_4+2p_5+3p_5+4p_5,$& $p_4+2p_5+3p_5+2p_8,$&$ 2p_5+2p_5+2p_5+p_8,$\\
 $ 2p_5+p_8+p_8+p_8,$&&\\
\end{tabular}
\end{center}
are universal over $\mathbb{Z}$.
\end{theorem}
\begin{proof}Using \eqref{varphi3Y1}, we obtain
\begin{equation}\label{Q20}
  \varphi(q^3)Y(q)Y^2(q^2)=X^2(q^4)Y^2(q^2)+qY^4(q^2).
\end{equation}
Since $3p_4+p_8+2p_8+2p_8$ is universal.  Thus \eqref{Q20}
yields the universality of the sums $2p_5+2p_5+p_8+p_8$ and $p_8+p_8+p_8+p_8.$
In a similar way, the universality $6p_4+p_8+p_8+2p_8$  with the help of \eqref{Y2} implies that, the sums
$3p_4+3p_4+p_8+p_8$ and $6p_3+3p_4+2p_5+p_8$ are universal. As $3p_4+3p_4+p_8+p_8$ is universal, we use \eqref{varphi3Y1}, to determine the
the universality of the sums $3p_3+p_5+2p_5+2p_5$ and $3p_3+p_5+p_8+p_8$. Similarly, since $2p_5+2p_5+p_8+p_8$
is universal.  Thus the identity \eqref{Y2} implies the universality of the sums $3p_4+p_5+p_5+p_8$ and $6p_3+p_5+p_5+2p_5$. Now,
since $3p_3+p_5+2p_5+2p_5 \sim 2p_5+2p_5+2p_5+p_8$ so that the sum $2p_5+2p_5+2p_5+p_8$ is also universal. Thus using \eqref{Y}, we deduce  that the sums
$p_5+p_5+p_5+4p_5$ and $p_5+p_5+p_5+2p_8$ are  universal over $\mathbb{Z}$.  Using \eqref{SunLemma1}, \eqref{WuSunequiv}, \eqref{SunLemma2} and \eqref{eqv1}, respectively, we find that
\begin{align*}
3p_3+3p_4+p_5+p_8 & \sim p_5+3p_5+6p_5+p_8  \sim p_3+p_5+6p_5+p_8\\
& \sim p_3+p_3+2p_3+6p_5\sim 2p_3+2p_3+p_4+6p_5. \end{align*}
Since  $3p_3+3p_4+p_5+p_8 $ is universal (by Theorem~\ref{thm1}), it follows that   $2p_3+2p_3+p_4+6p_5$ is also universal.
So that, the identity \eqref{varphi} yields the universality of the sums $p_3+p_3+2p_4+3p_5$ and $p_3+p_3+4p_3+3p_5$.
As $2p_3+2p_3+p_4+6p_5$ is universal and  $2p_3+2p_3+p_4+6p_5\sim 4p_3+p_4+2p_4+6p_5 \sim 2p_4+4p_5+6p_5+p_8,$  we conclude that, the sum
 $2p_4+4p_5+6p_5+p_8$ is universal and so \eqref{Y} yields the universality of the sums $p_4+2p_5+3p_5+4p_5$ and $p_4+2p_5+3p_5+2p_8$. Since
 $ 6p_3+3p_4+2p_5+p_8 \sim 3p_4+4p_5+p_8+2p_8$ and $ 6p_3+3p_4+2p_5+p_8 $ is universal, this means that the sum $3p_4+4p_5+p_8+2p_8$ is also universal.
 Thus using \eqref{varphi3Y1}, we deduce  that the sums $2p_5+2p_5+2p_5+p_8$ and $2p_5+p_8+p_8+p_8$ are  universal over $\mathbb{Z}$.  This completes the proof of the theorem.
\end{proof}
There are many  universal sums equivalent to those which we have determined in the previous theorems. In the following theorem, we choose some of them, where we use  \eqref{SunLemma1}--\eqref{WuSunequiv} to deduce and we omit the details.
\begin{theorem}\label{thm4} We have
\begin{align*}
&2p_5+4p_5+p_8+p_8 \sim 4p_3+p_4+2p_5+p_8 \sim 3p_3+4p_3+p_4+p_5\\
& \sim  2p_3+2p_4+p_8+p_8 \sim 3p_3+p_5+4p_5+p_8 \sim p_3+2p_3+3p_3+4p_5, \\
&3p_4+p_5+2p_5+p_8 \sim p_3+2p_3+3p_4+2p_5  \sim p_3+3p_4+2p_5+6p_5 \\
& \sim 3p_3+3p_4+p_5+p_5 \sim p_5+p_5+3p_5+6p_5  \sim p_3+p_5+p_5+6p_5,\\
& p_5+4p_5+p_8+p_8 \sim p_3+2p_3+4p_5+p_8 \sim  p_3+2p_3+4p_3+p_4\\
& \sim  2p_3+4p_3+p_5+2p_5  \sim p_5+2p_5+2p_5+2p_8 \sim p_3+p_4+2p_5+2p_8\\
 & \sim 4p_3+p_5+2p_5+6p_5 \sim p_3+4p_3+p_4+6p_5\sim p_3+4p_5+6p_5+p_8,\\
& 6p_3+p_5+2p_5+2p_5   \sim p_5+2p_5+4p_5+2p_8 \sim p_3+p_4+4p_5+2p_8 &\\
& \sim 2p_3+2p_4+p_5+2p_8 \sim 2p_3+4p_3+p_5+4p_5 \sim 2p_3+p_5+4p_5+12p_5,\\
&6p_3+p_5+2p_5+p_8 \sim p_3+2p_3+6p_3+2p_5 \sim p_3+6p_3+2p_5+6p_5\\
&\sim p_3+2p_5+6p_5+18p_5 \sim p_3+2p_3+4p_5+2p_8 \sim p_5+4p_5+p_8+2p_8\\
& \sim 4p_3+p_4+p_5+2p_8,\\
& 2p_5+4p_5+p_8+2p_8 \sim 2p_3+2p_4+p_8+2p_8 \sim  2p_3+4p_3+4p_5+p_8\\
& \sim 4p_3+p_4+2p_5+2p_8 \sim 2p_3+4p_5+12p_5+p_8\\
& 3p_3+3p_4+p_5+p_8  \sim p_5+3p_5+6p_5+p_8  \sim p_3+p_5+6p_5+p_8 \\
&\sim p_3+p_3+2p_3+6p_5 \sim 2p_3+2p_3+p_4+6p_5\sim 4p_3+p_4+2p_4+6p_5\\
&\sim 2p_4+4p_5+6p_5+p_8,\end{align*}
\begin{align*}
& 2p_5+2p_5+p_8+p_8  \sim 3p_3+p_5+2p_5+p_8 \sim 3p_3+3p_3+p_5+p_5\\
& \sim 6p_3+3p_4+p_5+p_5 \sim p_3+3p_3+p_4+p_8 \sim p_3+2p_3+3p_3+2p_5\\
& \sim p_3+3p_3+2p_5+6p_5,\\
& 3p_3+6p_3+p_5+2p_5 \sim p_5+2p_5+3p_5+3p_8   \sim p_3+p_4+3p_5+3p_8,\\
& 6p_3+3p_4+2p_5+p_8 \sim 3p_4+4p_5+p_8+2p_8 \sim 4p_3+p_4+3p_4+2p_8 ,\\
& 4p_5+p_8+2p_8+3p_8  \sim 4p_3+p_4+2p_8+3p_8,\\
& 4p_5+p_8+p_8+2p_8  \sim 4p_3+p_4+p_8+2p_8,\\
& 2p_5+p_8+2p_8+2p_8  \sim 2p_3+4p_3+p_8+2p_8,\\
& p_5+2p_5+4p_5+p_8  \sim 4p_3+p_4+p_5+2p_5,\\
& p_5+p_8+p_8+2p_8  \sim p_3+2p_3+p_8+2p_8,\\
& 4p_5+p_8+2p_8+5p_8  \sim 4p_3+p_4+2p_8+5p_8,\\
& 4p_5+p_8+p_8+3p_8  \sim 4p_3+p_4+p_8+3p_8,\\
& 4p_5+p_8+p_8+p_8 \sim 4p_3+p_4+p_8+p_8,\\
& 2p_5+p_8+p_8+2p_8  \sim 2p_3+4p_3+p_8+p_8, \\
& 3p_3+p_5+p_8+2p_8  \sim p_3+2p_3+3p_3+2p_8,\\
& 4p_5+p_8+2p_8+2p_8  \sim 4p_3+p_4+2p_8+2p_8,\\
& 6p_3+p_5+p_5+2p_5 \sim p_5+p_5+4p_5+2p_8,\\
& p_3+2p_5+4p_5+6p_5 \sim  p_3+2p_3+2p_4+6p_5,\\
& p_5+4p_5+p_8+p_8  \sim 4p_3+p_4+p_5+p_8,\\
& 4p_5+p_8+2p_8+6p_8  \sim 4p_3+p_4+2p_8+6p_8,\\
& p_3+6p_3+p_4+p_5  \sim 6p_3+p_4+p_5+3p_5.
\end{align*}
All are universal over $\mathbb{Z}$.
\end{theorem}

\ni 	\textbf{\large{Acknowledgements:}}\\
The work of the  third-named author (M. P. Chaudhary) was funded by the  National Board of Higher Mathematics (NBHM) of the Department of Atomic Energy (DAE) of the  Government of India by its sanction letter (Ref. No. 02011/12/ 2020 NBHM (R. P.)/R D II/7867) dated 19 October 2020.

\end{document}